\documentclass[11pt]{article}
\usepackage[margin=1in]{geometry} 
\usepackage{amssymb,amsfonts,amsmath,bbm,mathrsfs,stmaryrd}
\usepackage{xcolor}
\usepackage{url}

\usepackage{enumerate}

\usepackage{hyperref}
\hypersetup{colorlinks,
             linkcolor=black!75!red,
             citecolor=blue,
             pdftitle={},
             pdfproducer={pdfLaTeX},
             pdfpagemode=None,
             bookmarksopen=true
             bookmarksnumbered=true}

\usepackage{tikz}
\usetikzlibrary{arrows,calc,decorations.pathreplacing,decorations.markings,intersections,shapes.geometric,through,fit,shapes.symbols,positioning,decorations.pathmorphing}

\usepackage{braket}

\usepackage[amsmath,thmmarks,hyperref]{ntheorem}
\usepackage{cleveref}

\creflabelformat{enumi}{#2(#1)#3}

\crefname{section}{Section}{Sections}
\crefformat{section}{#2Section~#1#3} 
\Crefformat{section}{#2Section~#1#3} 

\crefname{subsection}{\S}{\S\S}
\AtBeginDocument{%
  \crefformat{subsection}{#2\S#1#3}%
  \Crefformat{subsection}{#2\S#1#3}%
}

\crefname{subsubsection}{\S}{\S\S}
\AtBeginDocument{%
  \crefformat{subsubsection}{#2\S#1#3}%
  \Crefformat{subsubsection}{#2\S#1#3}%
}

\theoremstyle{plain}

\newtheorem{lemma}{Lemma}[section]
\newtheorem{proposition}[lemma]{Proposition}
\newtheorem{corollary}[lemma]{Corollary}
\newtheorem{theorem}[lemma]{Theorem}

\theoremstyle{nonumberplain}
\newtheorem{theoremN}{Theorem}

\theoremstyle{plain}
\theorembodyfont{\upshape}
\theoremsymbol{\ensuremath{\blacklozenge}}

\newtheorem{definition}[lemma]{Definition}
\newtheorem{construction}[lemma]{Construction}
\newtheorem{example}[lemma]{Example}
\newtheorem{remark}[lemma]{Remark}
\newtheorem{remarks}[lemma]{Remarks}
\newtheorem{convention}[lemma]{Convention}
\newtheorem{notation}[lemma]{Notation}

\crefname{definition}{definition}{definitions}
\crefformat{definition}{#2definition~#1#3} 
\Crefformat{definition}{#2Definition~#1#3} 

\crefname{construction}{construction}{constructions}
\crefformat{construction}{#2construction~#1#3} 
\Crefformat{construction}{#2Construction~#1#3} 

\crefname{ex}{example}{examples}
\crefformat{example}{#2example~#1#3} 
\Crefformat{example}{#2Example~#1#3} 

\crefname{remark}{remark}{remarks}
\crefformat{remark}{#2remark~#1#3} 
\Crefformat{remark}{#2Remark~#1#3} 

\crefname{remarks}{remark}{remarks}
\crefformat{remarks}{#2remark~#1#3} 
\Crefformat{remarks}{#2Remark~#1#3} 

\crefname{convention}{convention}{conventions}
\crefformat{convention}{#2convention~#1#3} 
\Crefformat{convention}{#2Convention~#1#3} 

\crefname{notation}{notation}{notations}
\crefformat{notation}{#2notation~#1#3} 
\Crefformat{notation}{#2Notation~#1#3} 

\crefname{table}{table}{tables}
\crefformat{table}{#2table~#1#3} 
\Crefformat{table}{#2Table~#1#3}

\crefname{lemma}{lemma}{lemmas}
\crefformat{lemma}{#2lemma~#1#3} 
\Crefformat{lemma}{#2Lemma~#1#3} 

\crefname{proposition}{proposition}{propositions}
\crefformat{proposition}{#2proposition~#1#3} 
\Crefformat{proposition}{#2Proposition~#1#3} 

\crefname{corollary}{corollary}{corollaries}
\crefformat{corollary}{#2corollary~#1#3} 
\Crefformat{corollary}{#2Corollary~#1#3} 

\crefname{theorem}{theorem}{theorems}
\crefformat{theorem}{#2theorem~#1#3} 
\Crefformat{theorem}{#2Theorem~#1#3} 

\crefname{enumi}{}{}
\crefformat{enumi}{(#2#1#3)}
\Crefformat{enumi}{(#2#1#3)}

\crefname{assumption}{assumption}{Assumptions}
\crefformat{assumption}{#2assumption~#1#3} 
\Crefformat{assumption}{#2Assumption~#1#3} 

\crefname{claim}{claim}{Claims}
\crefformat{claim}{#2claim~#1#3} 
\Crefformat{claim}{#2Claim~#1#3} 

\crefname{equation}{}{}
\crefformat{equation}{(#2#1#3)} 
\Crefformat{equation}{(#2#1#3)}

\numberwithin{equation}{section}

\theoremstyle{nonumberplain}
\theoremsymbol{\ensuremath{\blacksquare}}

\newtheorem{proof}{Proof}
\newcommand\pf[1]{\newtheorem{#1}{Proof of \Cref{#1}}}

\newcommand\bG{{\mathbb G}}

\newcommand\bZ{{\mathbb Z}}

\newcommand\cA{{\mathcal A}}

\newcommand\cC{{\mathcal C}}

\newcommand\fg{{\mathfrak g}}
\newcommand\fh{{\mathfrak h}}

\newcommand\fk{{\mathfrak k}}

\newcommand\fn{{\mathfrak n}}
\newcommand\fo{{\mathfrak o}}

\newcommand\fy{{\mathfrak y}}
\newcommand\fz{{\mathfrak z}}

\newcommand\fsl{\mathfrak{sl}}
\newcommand\fsp{\mathfrak{sp}}

\DeclareMathOperator{\id}{id}
\DeclareMathOperator{\spn}{\mathrm{span}}
\DeclareMathOperator{\Ann}{\mathrm{Ann}}
\DeclareMathOperator{\End}{\mathrm{End}}

\DeclareMathOperator{\Ind}{\mathrm{Ind}}
\DeclareMathOperator{\tr}{\mathrm{tr}}

\DeclareMathOperator{\Prim}{\mathrm{Prim}}

\newcommand{\cat}[1]{\textsc{#1}}

\newcommand{\qedhere}{\mbox{}\hfill\ensuremath{\blacksquare}}

\title{Lie-algebra centers via de-categorification}
\author{Alexandru Chirvasitu}

\begin{document}

\date{}

\newcommand{\Addresses}{{
  \bigskip
  \footnotesize

  \textsc{Department of Mathematics, University at Buffalo, Buffalo,
    NY 14260-2900, USA}\par\nopagebreak \textit{E-mail address}:
  \texttt{achirvas@buffalo.edu}

}}

\maketitle

\begin{abstract}
  Let $\mathfrak{g}$ be a Lie algebra over an algebraically closed field $\Bbbk$ of characteristic zero. Define the universal grading group $\mathcal{C}(\mathfrak{g})$ as having one generator $g_{\rho}$ for each irreducible $\mathfrak{g}$-representation $\rho$, one relation $g_{\pi} = g_{\rho}^{-1}$ whenever $\pi$ is weakly contained in the dual representation $\rho^*$ (i.e. the kernel of $\pi$ in the enveloping algebra $U(\mathfrak{g})$ contains that of $\rho^*$), and one relation $g_{\rho} = g_{\rho'}g_{\rho''}$ whenever $\rho$ is weakly contained in $\rho'\otimes\rho''$.
  
  The main result is that attaching to an irreducible representation its central character gives an isomorphism between $\mathcal{C}(\mathfrak{g})$ and the dual $\mathfrak{z}^*$ of the center $\mathfrak{z}\le \mathfrak{g}$ when $\mathfrak{g}$ is (a) finite-dimensional solvable; (b) finite-dimensional semisimple. The group $\mathcal{C}(\mathfrak{g})$ is also trivial when the enveloping algebra $U(\mathfrak{g})$ has a faithful irreducible representation (which happens for instance for various infinite-dimensional algebras of interest, such as $\mathfrak{sl}(\infty)$, $\mathfrak{o}(\infty)$ and $\mathfrak{sp}(\infty)$). These are analogues of a result of M\"uger's for compact groups and a number of results by the author on locally compact groups, and provide further evidence for the pervasiveness of such center-reconstruction phenomena.
\end{abstract}


\vspace{.5cm}


\tableofcontents

\section*{Introduction}

The initial motivation for the material below is a phenomenon noted in \cite{mug}, which it will be instructive to summarize. Consider a compact group $\bG$, and define its {\it chain group} $\cC(\bG)$ by generators and relations, as follows:
\begin{itemize}
\item there is a generator $g_V$ for every irreducible unitary $\bG$-representation $V$;
\item and a relation $g_U = g_V g_W$ whenever $U$ is contained as a summand in $V\otimes W$.
\end{itemize}
The main result of \cite{mug} (namely \cite[Theorem 3.1]{mug}) says that assigning to an irreducible representation its central character implements an isomorphism of $\cC(\bG)$ onto the discrete abelian group $\widehat{Z(\bG)}$, i.e. the Pontryagin dual of the compact (abelian) center $Z(\bG)\le \bG$.

In other words, the dual center $\widehat{Z(\bG)}$ can be recovered from the category of $\bG$-representations by a process of {\it de-categorification} (hence this paper's title):
\begin{itemize}
\item objects (i.e. $\bG$-representations) are demoted to elements (of $\cC(\bG)$);
\item the tensor product bifunctor becomes multiplication;
\item and for good measure, dualization in the category of representations corresponds to taking inverses in $\cC(\bG)$: it turns out that the element $g_{V^*}$ corresponding to the dual (or contragredient) representation $V^*$ automatically equals $g_V^{-1}$. 
\end{itemize}
Note also the general ``Tannakian'' \cite{svdr,sch-tnk,dr-tnk,wor-tnk} flavor about the discussion: recovering structure from monoidal categorical data.

The theme is taken up in \cite{chi-cc} in the context of {\it locally} compact groups, where a chain group can be defined analogously (generators again given by irreducible unitary representations), with only the obvious sensible modifications: one imposes a relation $g_U = g_V g_W$ whenever $U$ is {\it weakly} contained in $V\otimes W$ in the sense of \cite[Definition F.1.1]{bdv} (actual containment would be too much to ask for).

The phenomenon turns out to be remarkably robust: one again has isomorphisms $\cC(\bG)\cong \widehat{Z(\bG)}$ of (this time topological) groups for broad interesting classes of locally compact groups $\bG$: discrete countable with infinite conjugacy classes, connected nilpotent Lie groups, connected semisimple Lie groups, etc.

The present iteration of the project investigates the natural purely algebraic analogues of the above-mentioned objects, constructions and results. Much of the discussion makes sense for Hopf algebras (in place of groups), and we do give definitions in that generality (\Cref{def:chainlie}), but the substance of the paper mostly concerns Lie algebras. Hence:

\begin{definition}\label{def:chainintro}
  Let $\fg$ be a Lie algebra over a field $\Bbbk$. The {\it chain group} $\cC(\fg)$ is defined by
  \begin{itemize}
  \item generators $g_{\rho}$ for irreducible $\fg$-representations $\rho$;
  \item and relations $g_{\rho} = g_{\rho'}g_{\rho''}$ whenever $\rho$ is {\it weakly contained} in $\rho'\otimes \rho''$ (\Cref{def:wcontlie}): the kernel of $\rho$, regarded as a morphism from the enveloping algebra $U(\fg)$, contains that of $\rho'\otimes\rho''$;
  \item together with relations $g_{\pi} = g_{\rho}^{-1}$ whenever $\pi$ is weakly contained in the contragredient representation $\rho^*$.
  \end{itemize}
\end{definition}

One might hope, by analogy to everything recalled above, that
\begin{itemize}
\item irreducible $\fg$-representations $\rho$ admit central characters, and in particular give functionals $\fz(\fg)\to \Bbbk$, where $\fz(\fg)\le \fg$ is the center of the Lie algebra (this is frequently the case, e.g. for finite-dimensional Lie algebras over algebraically closed fields of characteristic zero \cite[Proposition 2.6.8]{dixenv});
\item and that this then gives an isomorphism $\cC(\fg)\cong \fz(\fg)^*$, so that again, the dual center $\fz(\fg)^*$ is a de-categorification of the category of $\fg$-representations.
\end{itemize}
The main results of the paper confirm that this is holds in all cases I have been able to check (i.e. I do not know of any Lie algebras for which it is {\it not} true). Summarizing \Cref{pr:faiths}, \Cref{cor:infrk} and \Cref{th:solv,th:ss}:

\begin{theoremN}
  Let $\fg$ be a Lie algebra over an algebraically closed field $\Bbbk$ of characteristic zero and $\fz\le \fg$ its center.
  \begin{enumerate}[(a)]
  \item Associating to an irreducible representation its central character provides an isomorphism $\cC(\fg)\cong \fz^*$ if $\fg$ is finite-dimensional and either solvable or semisimple.
  \item Furthermore, $\cC(\fg)$ is trivial if the enveloping algebra $U(\fg)$ has a faithful simple module.
  \item So as a particular case of the previous item, we again have an isomorphism $\cC(\fg)\cong \fz^*$ (both sides being trivial) when $\fg$ is one of the infinite-rank complex Lie algebras $\fsl(\infty)$, $\fo(\infty)$ or $\fsp(\infty)$.  \qedhere
  \end{enumerate}
\end{theoremN}

\subsection*{Acknowledgements}

This work is partially supported by NSF grant DMS-2001128

Exchanges with M. Lorenz and I. Penkov on the present material and related matters has been very helpful and inspiring. 

\section{Preliminaries}\label{se.prel}

We make occasional reference to general ring-theoretic material for which \cite{lam,lam-lect}, say, are good reference. Coalgebras, bialgebras and Hopf algebras (over fields) will also feature sporadically; \cite{swe,mont,rad} all provide good background, which we reference with more specificity where appropriate. All rings are assumed unital.

The phrase `representation' will be employed as a synonym for `left module', appropriately linear when working over a field. Additionally, modules are left unless specified otherwise.

Recall that a proper two-sided ideal $I\trianglelefteq A$ of a ring is
\begin{itemize}
\item {\it primitive} (e.g. \cite[Definitions 11.2 and 11.3]{lam}, \cite[\S 3.1.4]{dixenv}) if it is the kernel of an irreducible $A$-representation.
\item {\it prime} (\cite[Definition 10.1]{lam}, \cite[\S 3.1.1]{dixenv}) if for ideals $J_i$, $i=1,2$ with $J_1 J_2\le I$, $I$ must contain one of the $J_i$.
\item and {\it semi-prime} (\cite[Definition 10.8]{lam}, \cite[\S 3.1.3]{dixenv}) if for any ideal $J$ with $J^2\le I$ we have $J\le I$.
\end{itemize}

We apply the term `primitive' universally, whether in the purely algebraic or analytic setting (where algebras are $C^*$, ideals are closed, representations are on Hilbert spaces with the appropriate notion of irreducibility, etc.).

The following construction features prominently in the discussion below.

\begin{definition}\label{def:cmon}
  Consider a set $(S,\triangleleft)$ equipped with a ternary relation, written
  \begin{equation*}
    s\triangleleft (s',s'').
  \end{equation*}
  \begin{itemize}
  \item The {\it chain semigroup} $\cC=\cC(S,\triangleleft)$ associated to `$\triangleleft$' is defined as having one generator $g_s$ for each $s\in S$ and relations
    \begin{equation*}
      g_s = g_{s'}g_{s''}
      \quad\text{whenever}\quad
      s\triangleleft (s',s''). 
    \end{equation*}
    It will occasionally be convenient to enrich the structure of $\cC$ as follows.
  \item If in addition $S$ is equipped with a distinguished element $s_0$, the {\it chain monoid} $\cC=\cC(S,\triangleleft,s_0)$ is defined as above, with the additional constraint that $g_{s_0}$ be the trivial element of the monoid.
  \item Finally, if $S$ is also equipped with a binary relation `$\sim$', the {\it chain group} $\cC=\cC(S,\triangleleft,s_0,\sim)$ is the chain monoid, with the additional constraint that
    \begin{equation*}
      g_{s'} = g_{s}^{-1}
      \quad\text{whenever}\quad
      s\sim s'.
    \end{equation*}
  \end{itemize}
\end{definition}

\section{Chain groups and center reconstruction}\label{se:cg}

We will soon specialize the discussion to Lie algebras, but some of it goes through more generally. $\Bbbk$ always denotes a field that algebras, Hopf algebras, Lie algebras and so on are understood to be linear over. Additional assumptions will be in force throughout most of the paper, but the reader will be warned when they come into effect.

\subsection{Generalities on weak containment for Hopf and Lie algebras}\label{subse:wc}

For an arbitrary ring $A$, $\Prim(A)$ is its space of primitive ideals It can be equipped with the familiar {\it Jacobson topology} (\cite[\S 3.2.2]{dixenv} or \cite[\S 7.1.3]{muss}): for an ideal $I\trianglelefteq A$, set
\begin{equation*}
  V(I):=\{P\in \Prim(A)\ |\ I\le P\}.
\end{equation*}
The $V(I)$ are precisely the closed sets of the topology. For a Lie algebra $\fg$ with enveloping algebra $U=U(\fg)$ the notation $\Prim(\fg):=\Prim(U)$ is an alternative.

We also borrow the usual language of {\it weak containment} from functional analysis \cite[\S 3.4.5]{dixc}; 

\begin{definition}\label{def:wcontlie}
  Let $\pi$ and $\rho$ two representations of a ring $A$.
  \begin{itemize}
  \item We say that $\pi$ is {\it weakly contained} in $\rho$ (written $\pi\preceq \rho$) if $\ker\pi\supseteq \ker\rho$.
  \item $\pi$ and $\rho$ are {\it weakly equivalent} (written $\pi\approx \rho$) if each of them weakly contains the other. 
  \end{itemize}
  The terms apply to Lie-algebra representations, where $A$ is taken to be the enveloping algebra.

  When convenient, we might also commit mild notational abuse in writing $M\preceq N$ if $M$ and $N$ are the left $A$-modules carrying the two respective representations.
\end{definition}

\begin{remark}\label{re:weqsimples}
  When the kernel of a representation $\rho:A\to \End(V)$ is an intersection of primitive ideals, we have a weak equivalence
  \begin{equation*}
    \rho\approx \bigoplus \pi,\ \text{irreducible }\pi\preceq \rho.
  \end{equation*}
  When $A=U(\fg)$ is the envelope of a finite-dimensional Lie algebra, this is the case precisely if $\ker\rho$ is semi-prime \cite[Proposition 3.1.15]{dixenv}.
\end{remark}

There are alternative characterizations of weak containment, parallel to their analytic counterparts (\cite[Theorem 3.4.4]{dixc}, \cite[Theorem F.4.4]{bdv}). First, we need

\begin{definition}\label{def:wast}
  Let $V$ be a $\Bbbk$-vector space. The {\it weak$^*$ topology} on $V^*$ is the weakest topology making all maps
  \begin{equation*}
    V^*\ni f\mapsto f(v)\in \Bbbk,\ v\in V
  \end{equation*}
  continuous (with $\Bbbk$ topologized discretely).
\end{definition}

There is also the following notion (following \cite[\S 2.7.8]{dixenv}, for instance).

\begin{definition}\label{def:coeff}
  Let $\rho:A\to \End(V)$ be a representation of a $\Bbbk$-algebra on a vector space. The space $MC(\rho)$ of {\it matrix coefficients} (or just plain `coefficients') of $\rho$ is
  \begin{equation*}
    MC(\rho):=\spn\{f(\rho(\cdot)v)\ |\ v\in V,\ f\in V^*\}\le A^*. 
  \end{equation*}  
\end{definition}

We follow standard practice (e.g. \cite[\S 1.2.20]{dixenv}) in denoting, with a `$\perp$' superscript, annihilators of vector spaces with respect to a pairing/bilinear form. Specifically, if
\begin{equation*}
  W\otimes V\xrightarrow[]{\quad b\quad} \Bbbk
\end{equation*}
is such a pairing and $V_0\le V$, 
\begin{equation*}
  V_0^{\perp}:=\{w\in W\ |\ b(w,V_0)=\{0\}\}.
\end{equation*}
The pairing will always be understood, and in fact it will typically be the standard one between a vector space $V$ and its full dual $V^*$.

We now have the following simple analogue of \cite[Theorem 3.4.4]{dixc}. 

\begin{lemma}\label{le:wastdense}
  For representations $\rho_i:A\to \End(V_i)$, $i=1,2$ of an algebra, the following conditions are equivalent.
  \begin{enumerate}[(a)]
  \item\label{item:4} $\rho_1\preceq\rho_2$ in the sense of \Cref{def:wcontlie}.
  \item\label{item:5} We have the inclusion
    \begin{equation*}
      \overline{MC(\rho_1)}\le \overline{MC(\rho_2)}
    \end{equation*}
    in $A^*$, with the bar denoting the weak$^*$ closure and $MC$ as in \Cref{def:coeff}.
  \item\label{item:6} Similarly, we have the inclusion
    \begin{equation*}
      MC(\rho_2)^{\perp} \le MC(\rho_1)^{\perp}
    \end{equation*}
    in $A$. 
  \end{enumerate}
\end{lemma}
\begin{proof}
  The central observations are:
  \begin{enumerate}[(I)]
  \item\label{item:2} For any vector space $V$, if
    \begin{equation*}
      \id:\End(V)\to \End(V)
    \end{equation*}
    is the standard representation of its endomorphism algebra, $MC(\id)$ is weak$^*$-dense in $\End(V)^*$.
  \item\label{item:3} For any $W\le V^*$, the weak$^*$ closure of $W$ is nothing but the annihilator of $W^{\perp}\le V$, i.e. $W^{\perp\perp}$. 
  \end{enumerate}
  This latter remark immediately implies the equivalence of \Cref{item:5} and \Cref{item:6}, while \Cref{item:2} shows that \Cref{item:6} reads
  \begin{equation*}
    \ker\rho_2\le \ker\rho_1
  \end{equation*}
  and hence is equivalent to \Cref{item:4}. 
\end{proof}

We use the language of {\it induced representations} of \cite[Chapter 5]{dixenv}.

\begin{notation}\label{not:induced}
  Let $\fh\le \fg$ be an inclusion of Lie algebras, and denote the enveloping-algebra construction by $U(\cdot)$.
  \begin{enumerate}[(a)]
  \item For a representation $\rho:U(\fh)\to \End(W)$ the corresponding {\it induced representation}
    \begin{equation*}
      \Ind(\rho)
      \quad\text{or}\quad
      \Ind^{\fg}(\rho)
      \quad\text{or}\quad
      \Ind_{\fh}^{\fg}(\rho)
    \end{equation*}
    is $U(\fg)\otimes_{U(\fh)}W$, with its obvious left $U(\fg)$-module structure.
  \item Assume now that $\fg$ is finite-dimensional.
    \begin{itemize}
    \item In general, for an inclusion $F\le E$ of finite-dimensional vector spaces left invariant by an operator $T\in\End(E)$, write $\tr_{E/F}(T)$ for the trace of the operator induced by $T$ on the quotient space $E/F$.
    \item For $x\in \fh$ set
      \begin{equation*}
        \theta_{\fg,\fh}(x):=\frac 12 \tr_{\fg/\fh}ad(x),
      \end{equation*}
      where $ad:\fg\to \End(\fg)$ is the adjoint representation.
    \item For a representation $\rho:\fh\to \End(V)$, its {\it twist} $\widetilde{\rho}$ is defined by
      \begin{equation*}
        \fh\ni x\xrightarrow[]{\quad \widetilde{\rho}\quad}\rho(x)+\theta_{\fg,\fh}(x)\id \in \End(V).
      \end{equation*}
      Regarding the functional $\theta_{\fg,\fh}\in \fh^*$ as a 1-dimensional $\fh$-representation (which it is, since it vanishes on $[\fh,\fh]$), we also have
      \begin{equation*}
        \widetilde{\rho}\cong \rho\otimes \theta_{\fg,\fh}. 
      \end{equation*}
    \item Finally, {\it twisted induction} $\widetilde{\Ind}$ is defined by
      \begin{equation*}
        \widetilde{\Ind}_{\fh}^{\fg} (\rho)
        :=
        \Ind_{\fh}^{\fg} (\widetilde{\rho}).
      \end{equation*}
    \end{itemize}
  \end{enumerate}
\end{notation}

\begin{remark}\label{re:0twist}
  When $\fg$ is nilpotent all traces making a difference between induction and twisted induction in \Cref{not:induced} vanish, so in that case there is no distinction between $\Ind$ and $\widetilde{\Ind}$. 
\end{remark}

As in the analytic case, the familiar operations of tensoring, induction, etc. respect weak containment (see e.g. \cite[\S F.3]{bdv} for the versions pertaining to locally compact groups).

\begin{proposition}\label{pr:opscont}
  The weak containment relation $\preceq$ of \Cref{def:wcontlie} is compatible with the following operations.
  \begin{enumerate}[(a)]
  \item\label{item:7} Direct sums, in the sense that if $\pi_i\preceq \rho_i$ for a family $i\in I$ of representations of a ring $A$, we also have
    \begin{equation*}
      \bigoplus_I \pi_i\preceq \bigoplus_I \rho_i.
    \end{equation*}
  \item\label{item:8} Tensor products, for bialgebras over fields: 
    \begin{equation*}
      \pi_i\preceq \rho_i,\ i=1,2
      \Rightarrow
      \pi_1\otimes \pi_2\preceq \rho_1\otimes\rho_2
    \end{equation*}
    for representations $\pi_i$ and $\rho_i$ of a bialgebra $H$ over any field.
  \item\label{item:19} Scalar extension (or induction), for free ring extensions: suppose $A\to B$ is a ring embedding with $B/A$ free as a right $A$-module. If $M$ and $N$ are two left $A$-modules,
    \begin{equation*}
      M\preceq N
      \Longrightarrow
      B\otimes_A M \preceq B\otimes_A N. 
    \end{equation*}
  \item\label{item:9} Induction, for Lie algebras: if $\fh\le \fg$ is an inclusion of Lie algebras and $\pi\preceq \rho$ are $\fh$-representations, then
    \begin{equation*}
      \Ind^{\fg}(\pi)\preceq \Ind^{\fg}(\rho). 
    \end{equation*}
  \item\label{item:10} Twisted induction: same as in \Cref{item:9}, but with $\widetilde{\Ind}$ in place of $\Ind$ (assuming $\fg$ is finite-dimensional).
  \end{enumerate}
\end{proposition}
\begin{proof}
  Considering the claims in the stated order:

  \begin{enumerate}[(a)]
  \item follows from the fact that
    \begin{equation*}
      \ker\left(\bigoplus_I \pi_i\right) = \bigcap_I \ker\pi_i
    \end{equation*}
    and similarly for the $\rho_i$, so if $\ker\pi_i\ge \ker\rho_i$ the same goes for $I$-fold intersections.
  \item Consider a bialgebra $H$ as in the statement. We then have
    \begin{equation*}
      \ker(\pi_1\otimes\pi_2) = \ker\pi_1\wedge \ker\pi_2,
    \end{equation*}
    where
    \begin{equation*}
      V\wedge W:=\ker\left(H\xrightarrow[]{\quad\Delta\quad} H\otimes H\xrightarrow[]{} (H/V)\otimes (H/W)\right)
    \end{equation*}
    is the {\it wedge product} of two subspaces $V,W\le H$ (\cite[p.179]{swe} or \cite[proof of Theorem 5.2.2]{mont}). That ``wedging'' respects inclusion is elementary, hence the conclusion.
  \item Let $J_M$ and $J_N$ be the respective kernels of the module-structure maps
    \begin{equation*}
      A\to \End(M)
      \quad\text{and}\quad
      A\to \End(N),
    \end{equation*}
    the freeness assumption ensures that the proof of \cite[Proposition 5.1.7 (i)]{dixenv} replicates to show that
    \begin{equation*}
      \Ann_{B}(M) = BJ_M
    \end{equation*}
    and similarly for $N$, where
    \begin{equation*}
      M\le B\otimes_AM,\ N\le B\otimes_A N
    \end{equation*}
    via the canonical maps. But as in \cite[Proposition 5.1.7 (ii)]{dixenv}, the annihilators of $M\le B\otimes_AM$ in $B$ are the largest two-sided ideals contained in
    \begin{equation*}
      \Ann_{B}(M) = BJ_M
      \quad\text{and}\quad
      \Ann_{B}(N) = BJ_N:
    \end{equation*}
    this is an instance of the general remark that for any set $S$ of generators for a left $R$-module $V$, the annihilator of $V$ is the largest bilateral ideal contained in the annihilator of $S$.
    
    The assumption
    \begin{equation*}
      J_M\ge J_N\Longrightarrow BJ_M\ge BJ_N
    \end{equation*}
    then implies the same ordering between the largest 2-sided ideals respectively contained in the two, and we are done. 
    
  \item is a consequence of \Cref{item:19}, applied to the ring inclusion $U(\fh)\le U(\fg)$, the freeness assumption being a consequence of the {\it Poincar\'e-Birkhoff-Witt theorem} (\cite[Theorem 2.1.11]{dixenv} or \cite[Theorem 6.1.1]{muss}).
  \item follows from points \Cref{item:8} and \Cref{item:9}, given the definition
    \begin{equation*}
      \widetilde{\Ind}^{\fg}(\pi) = \Ind^{\fg}(\pi\otimes\theta_{\fg,\fh})
    \end{equation*}
    of \Cref{not:induced} and its $\rho$ analogue.
  \end{enumerate}

  This concludes the proof.
\end{proof}

\begin{remarks}
  \begin{enumerate}[(1)]
  \item An alternative proof of \Cref{pr:opscont} \Cref{item:8} could have used \Cref{le:wastdense}: we are assuming
    \begin{equation*}
      \overline{MC(\pi_i)}\le \overline{MC(\rho_i)},
    \end{equation*}
    whence also
    \begin{equation*}
      \overline{MC(\pi_1)\otimes MC(\pi_2)}\le \overline{MC(\rho_1)\otimes MC(\rho_2)}. 
    \end{equation*}
    The conclusion then follows from the fact that in general, $V^*\otimes W^*$ is weak$^*$-dense in $(V\otimes W)^*$.
  \item In reference to \Cref{pr:opscont} \Cref{item:19}, the issue of weak-containment permanence under scalar extension is a bit delicate. On the one hand that statement made a fairly strong freeness assumption. On the other hand though, even {\it faithful flatness} \cite[\S 4I]{lam-lect} of $B$ as a right $A$-module would not quite have been sufficient, as \Cref{ex:2prime} shows. 
  \end{enumerate}
\end{remarks}

\begin{example}\label{ex:2prime}
  Let $p\ne q$ be two prime numbers, and denote by $A$ the {\it localization} \cite[Chapter 3]{am} of $\bZ$ away from the prime ideals $(p)$ and $(q)$: the ring obtained from $\bZ$ by inverting all primes distinct from $p$ and $q$.

  We consider two $A$ modules:
  \begin{equation*}
    {}_pM:=\bigoplus_{n} A/p^nA
  \end{equation*}
  and similarly for ${}_qM$ (with $q$ in place of $p$). Both are faithful, in the sense that their annihilators are trivial. In particular, ${}_pM\approx {}_qM$ (\Cref{def:wcontlie}).

  The ring extension $A\to B$ will now be the {\it $(pq)$-adic completion} of \cite[discussion following Proposition 10.5]{am}:
  \begin{equation*}
    B:=\varprojlim_n A/(pq)^n A. 
  \end{equation*}
  That $A\to B$ is faithfully flat follows, say, from \cite[Chapter 10, Exercise 7]{am}. But $B$ is easily computed to be the product
  \begin{equation*}
    B\cong \bZ_p\times \bZ_q
  \end{equation*}
  of the rings of {\it $p$-adic} and {\it $q$-adic integers} (\cite[p.105, Example 2]{am}) respectively. Because $p$ is invertible in $\bZ_q$, the ideal
  \begin{equation*}
    \bZ_q\trianglelefteq B\cong \bZ_p\times \bZ_q
  \end{equation*}
  annihilates ${}_pM$ and hence $B\otimes_A {}_pM$. With just a trace amount of additional effort one can show that in fact
  \begin{equation*}
    \Ann_N(B\otimes_A {}_pM) = \bZ_q
    \quad\text{and}\quad
    \Ann_N(B\otimes_A {}_qM) = \bZ_p.
  \end{equation*}
  In particular, the two modules have annihilators that are incomparable under containment, so the weak containment relation has not survived the faithfully flat scalar extension along $A\to B$.
\end{example}

\begin{definition}\label{def:chainlie}
  \begin{enumerate}[(1)]
  \item Let $H$ be a Hopf algebra over a field $\Bbbk$. The {\it chain group} $\cC(H)$ is that of \Cref{def:cmon}, for
    \begin{itemize}
    \item the set $S$ of isomorphism classes of simple $H$-modules;
    \item the ternary relation
      \begin{equation}\label{eq:rrr}
        \rho\triangleleft (\rho',\rho'')
        \Longleftrightarrow
        \rho\preceq \rho'\otimes\rho'';
      \end{equation}
    \item the distinguished element of $S$ corresponding to the trivial $H$-module $\Bbbk$ induced by the counit $\varepsilon:H\to \Bbbk$;
    \item and the binary relation `$\sim$' defined by
      \begin{equation*}
        \rho \sim \text{ any irreducible representation weakly contained in the dual }\rho^*.
      \end{equation*}
    \end{itemize}
  \item Similarly, the chain group $\cC(\fg)$ of a Lie algebra $\fg$ is that of its universal enveloping algebra with its usual Hopf-algebra structure \cite[Example 1.5.4]{mont}: $\cC(\fg):=\cC(U(\fg))$. 
  \end{enumerate}
\end{definition}

\begin{remarks}\label{res:primenough}
  \begin{enumerate}[(1)]
  \item\label{item:13} References in \Cref{def:chainlie} to primitive ideals containing arbitrary ideals are unproblematic, so that $\cC(H)$ is indeed a well-defined group: every proper ideal is contained in a maximal one (for any unital ring, by Zorn's lemma), and in turn maximal ideals are primitive \cite[\S 3.1.6]{dixenv}.
  \item\label{item:11} Suppose the Lie algebra $\fg$ of \Cref{def:chainlie} is finite-dimensional, and let $\rho$ be an irreducible $\fg$-representation.
    
    The kernel of $\rho^*$ is easily seen to be $S(\ker\rho)$, where $S:U(\fg)\to U(\fg)$ is the antipode. Since $S$ is an anti-automorphism (the {\it principal anti-automorphism} $x\mapsto x^T$ of \cite[\S 2.2.18]{dixenv}), $\ker\rho^*$ is semi-prime and thus an intersection of primitive ideals (\Cref{re:weqsimples}). But this means that
    \begin{equation*}
      \rho^*\approx \bigoplus \pi
    \end{equation*}
    for irreducible $\pi\preceq \rho^*$ as in \Cref{re:weqsimples}. There are, in particular, ``enough'' irreducible representations that play the role of the inverse to $\rho$ in \Cref{def:chainlie}.
  \item\label{item:14} And in fact, if furthermore $\fg$ is solvable (as well as finite-dimensional), the primitive ideals of its enveloping algebra $U:=U(\fg)$ are precisely those prime ideals $I\trianglelefteq U$ for which the intersection of the primes $I'\supsetneq I$ contains $I$ strictly \cite[Theorem 4.5.7]{dixenv}.
    
    It follows from this characterization that any anti-automorphism of $U$ sends primitive ideals to primitive ideals, and hence $\ker \rho^*$ is {\it primitive} (as opposed to just (semi-)prime). The same remark is made in \cite[Chapitre I, \S 8]{duflo-nil}, along with related comments.
  \item\label{item:12} The generator $g_{\rho}\in \cC(H)$ of the chain group does not actually depend on the isomorphism class of $\rho$, but rather only on the primitive ideal $\ker\rho$. This is immediate from \Cref{pr:opscont} \Cref{item:8}, which implies that
    \begin{equation*}
      \rho\preceq \rho'\otimes\rho''
    \end{equation*}
    entails the same relation upon substituting for $\rho''$ (say) any irreducible representation weakly equivalent to it (i.e. having the same kernel).

    We will often take this observation for granted in the sequel, and extend the notation $g_{\rho}$ (for irreducible representations $\rho$) to $g_J$ (for primitive ideals $J=\ker \rho$).
  \end{enumerate}
\end{remarks}

In reference to \Cref{res:primenough} \Cref{item:12}, we observe that passing to the chain group obliterates the distinction between primitive ideals ordered by inclusion:

\begin{lemma}\label{le:incsame}
  Let $J\le J'\trianglelefteq H$ be two primitive ideals of a Hopf algebra. In the notation of \Cref{res:primenough} \Cref{item:12}, we have
  \begin{equation*}
    g_J = g_{J'}\text{ in }\cC(H). 
  \end{equation*}
\end{lemma}
\begin{proof}
  Let $\rho$ and $\rho'$ be irreducible representations with kernels $J$ and $J'$ respectively. We have $\rho\preceq \rho'$ by assumption, so that
  \begin{equation*}
    \rho\otimes\rho'\preceq \rho'\otimes\rho'
  \end{equation*}
  by \Cref{pr:opscont} \Cref{item:8}. But then irreducible representations weakly contained in the left-hand side are also weakly contained in the right-hand side, whence
  \begin{equation*}
    g_J g_{J'} = g_{\rho} g_{\rho'} = g_{\rho'} g_{\rho'} = g_{J'} g_{J'}
  \end{equation*}
  in $\cC(\fg)$. Since the latter is a group, this indeed implies $g_J = g_{J'}$. 
\end{proof}

As a simple consequence, chain groups are easy to understand for Hopf algebras with a ``large'' simple representation.

\begin{proposition}\label{pr:faiths}
  The chain group $\cC(H)$ of a Hopf algebra with a faithful simple module is trivial.
\end{proposition}
\begin{proof}
  Since $\{0\}$ is primitive and contained in any other primitive ideal $J\trianglelefteq H$, \Cref{le:incsame} says that $g_{J}=g_{\{0\}}$ for all $J$. In short, the chain group is a singleton (and hence trivial).
\end{proof}

In particular, \Cref{pr:faiths} applies to infinite-dimensional Lie algebras that have gained some recent attention: the infinite-rank $\fsl(\infty)$, $\fo(\infty)$ and $\fsp(\infty)$ obtained by embedding $\fsl(n)\subset \fsl(n+1)$ in the obvious fashion, as upper left-hand corner matrices (and similarly for orthogonal and symplectic matrices). The primitive ideals of their enveloping algebras are classified in \cite{pp3,pp4}, and as observed in \cite[Introduction]{pp3}, these enveloping algebras have (many) faithful simple modules; consequently:

\begin{corollary}\label{cor:infrk}
  The infinite-rank classical complex Lie algebras $\fsl(\infty)$, $\fo(\infty)$ and $\fsp(\infty)$ of \cite[\S 1]{pp3} have trivial chain groups.  \qedhere
\end{corollary}

The significance of this remark in the present context will become apparent later, when we seek to identify (typically for {\it finite}-dimensional Lie algebras) the chain group $\cC(\fg)$ with the dual $\fz^*$ of the center $\fz\le \fg$. Because $\fsl(\infty)$, $\fo(\infty)$ and $\fsp(\infty)$ have trivial centers, \Cref{cor:infrk} fits into the same pattern as \Cref{th:solv,th:ss} below.

\begin{remark}
  Although \Cref{cor:infrk} focuses on a few specific infinite-dimensional Lie algebras, {\it finite}-dimensional examples fitting into the mold of \Cref{pr:faiths} exist: according to \cite[Theorem 6.1.1 and Lemma 6.1.2 (i)]{dixenv}, for instance, the non-abelian 2-dimensional Lie algebra has the requisite property.
\end{remark}

\begin{convention}
  We henceforth focus on
  \begin{itemize}
  \item finite-dimensional Lie algebras;
  \item over algebraically closed fields $\Bbbk$ of characteristic zero.
  \end{itemize}
  Unless specified otherwise (e.g. in reverting to the general case by explicitly mentioning `arbitrary fields' or some such phrase), these assumptions are in place throughout the sequel.
\end{convention}

The next few subsections are titled for the various classes of Lie algebras under consideration therein.

\subsection{Nilpotent}\label{subse:nil}

Assume until further notice that $\fg$ is finite-dimensional. Because furthermore the ground field is algebraically closed, every irreducible representation has a {\it central character} \cite[\S 2.6.7 and Proposition 2.6.8]{dixenv}: the center $Z(\fg)$ of the enveloping algebra $U(\fg)$ acts via an algebra morphism $\chi:Z(\fg)\to \Bbbk$. In particular, the center $\fz(\fg)$ of $\fg$ acts via a linear functional
\begin{equation*}
  \chi|_{\fz(\fg)}\in \fz(\fg)^*;
\end{equation*}
this provides a canonical group morphism
\begin{equation}\label{eq:canlie}
  \cat{can}=\cat{can}_{\fg}:\cC(\fg)\to \fz(\fg)^*,
\end{equation}
and it will be of interest to determine when/whether this map is a group isomorphism. Its additivity is clear, and moreover surjectivity is unproblematic:

\begin{lemma}\label{le:cansurj}
  For any finite-dimensional Lie algebra $\fg$ over the algebraically closed field $\Bbbk$ the morphism \Cref{eq:canlie} is onto.
\end{lemma}
\begin{proof}
  Consider an arbitrary functional $f$ on the center $\fz:= \fz(\fg)$, regarded as a 1-dimensional representation of $\fz$. Any irreducible $\fg$-representation
  \begin{equation*}
    \rho\preceq \Ind_{\fz}^{\fg}(f)
  \end{equation*}
  (such representations do exist by \Cref{res:primenough} \Cref{item:13}) will then be acted upon by $\fz$ via $f$, so $f$ is the image of $g_{\rho}$ through \Cref{eq:canlie}.
\end{proof}

The main result to be discussed here is

\begin{theorem}\label{th:nil}
  Let $\fg$ be a finite-dimensional nilpotent Lie algebra over the algebraically closed field $\Bbbk$ of characteristic zero.

  The canonical map \Cref{eq:canlie} is an isomorphism.
\end{theorem}

We need some preparation. First, the following simple observation is an analogue of sorts (albeit a much less precise and powerful one) for the usual induction-restriction result on unitary representations of locally compact groups \cite[Theorem 12.1]{mack-ind1}.

\begin{lemma}\label{le:indres}
  Let $\fh,\fg'\le \fg$ be inclusions of Lie algebras and $\rho:\fh\to \End(W)$ an $\fh$-representation. We then have the inclusion
  \begin{equation}\label{eq:tworeps}
    \ker\Ind_{\fh}^{\fg}(\rho)|_{\fg'}
    \le 
    \ker\Ind_{\fh\cap \fg'}^{\fg'}(\rho|_{\fh\cap \fg'})
  \end{equation}
\end{lemma}
\begin{proof}
  We write $W_{\fg}$ and $W_{\fg'}$ for the carrier spaces of the two representations
  \begin{equation*}
    \Ind_{\fh}^{\fg}(\rho)|_{\fg'}
    \quad\text{and}\quad
    \Ind_{\fh\cap \fg'}^{\fg'}(\rho|_{\fh\cap \fg'})
  \end{equation*}
  respectively, and denote by $J\le U(\fh)$ the kernel of $\rho$.

  According to \cite[Proposition 5.1.7 (i)]{dixenv}, the annihilator of $W\le W_{\fg}$ in $U(\fg)$ is the left ideal $U(\fg)J$. Choose an ordered basis for $\fg$ consisting, in this order, of
  \begin{itemize}
  \item a basis for a subspace of $\fg$ supplementing $\fh+\fg'$;
  \item a basis for a subspace of $\fg'$ supplementing $\fh\cap \fg'$;
  \item one for $\fh\cap \fg'$;
  \item and finally, one for a subspace of $\fh$ supplementing $\fh\cap \fg'$.
  \end{itemize}
  Plugging that basis into the PBW theorem (\cite[Theorem 2.1.11]{dixenv} or \cite[Theorem 6.1.1]{muss}), it will follow easily that an element of $U(\fg')$ annihilates $W\le W_{\fg}$ if and only if it belongs to the left ideal
  \begin{equation}\label{eq:ufgdot}
    U(\fg')\cdot\left(\text{kernel of }\rho_{\fh\cap \fg'}\right).
  \end{equation}
  But that means that the kernel of the left-hand side of \Cref{eq:tworeps} is a bilateral ideal of $U(\fg')$ contained in the right-hand side of \Cref{eq:ufgdot}, whereas by \cite[Proposition 5.1.7 (ii)]{dixenv} the right-hand side of \Cref{eq:tworeps} is the {\it largest} such ideal. The inclusion \Cref{eq:tworeps} follows. 
\end{proof}

Recall now, briefly, how the classification of primitive ideals for solvable Lie algebras proceeds (the process is summarized in \cite[\S 6.1.5]{dixenv}); this will also serve to fix some notation.

\begin{construction}\label{con:primsolv}
  Throughout, $\fg$ is assumed solvable (as always, over an algebraically-closed characteristic-0 field).
  \begin{itemize}
  \item Consider an element $f\in \fg^*$, i.e. a linear functional on $\fg$.
  \item To it, associate any {\it polarization} $\fh\le \fg$; recall \cite[1.12.8]{dixenv} that this means
    \begin{itemize}
    \item $\fh$ is {\it subordinate} to $f$ in the sense that $f|_{[\fh,\fh]}\equiv 0$;
    \item and its dimension achieves the theoretical maximum:
      \begin{equation*}
        \dim\fh = \frac 12\left(\dim\fg + \dim\fg^f\right),
      \end{equation*}
      where
      \begin{equation*}
        \fg^f:=\{x\in \fg\ |\ f([x,y])=0,\ \forall y\in \fg\}
      \end{equation*}
      is the Lie subalgebra of $\fg$ leaving $f$ invariant under the coadjoint action.
    \end{itemize}    
    Polarizations always exist for solvable Lie algebras over algebraically closed fields \cite[Proposition 1.12.10]{dixenv}. 
  \item Being $f$-subordinate, $\fh$ carries a one-dimensional representation induced by $f$; we denote it by the same symbol (i.e. `$f$').
  \item Now form the twisted induced representation $\widetilde{\Ind}_{\fh}^{\fg}(f)$ as in \Cref{not:induced}.
  \item Set
    \begin{equation*}
      I(f):=\ker \widetilde{\Ind}_{\fh}^{\fg}(f).
    \end{equation*}
    This turns out to be a primitive ideal of $U(\fg)$.
  \item That primitive ideal does not actually depend on the polarization $\fh$ or indeed even on $f$ itself, but only on the orbit of $f$ under the action of the algebraic adjoint group $\cA$ attached to $\fg$.
  \item And the resulting map
    \begin{equation*}      
      \overline{I}:\fg^*/\cA \to \Prim(U)
    \end{equation*}
    is bijective.
  \end{itemize}
\end{construction}

\begin{notation}\label{not:irrif}
  As a matter of convenience, we occasionally write $\rho_{f}$ for an irreducible representation with kernel $I(f)$ as in \Cref{con:primsolv}. Such a representation is in general not unique subject to this condition, but this will not matter whenever the notation is in use.
\end{notation}

In reference to all of this, we now have

\begin{lemma}\label{le:resnil}
  Let $\fg'\le \fg$ be an inclusion of finite-dimensional nilpotent Lie algebras and $f\in \fg^*$. In the notation of \Cref{con:primsolv}, we have the inclusion
  \begin{equation}\label{eq:resnil}
    I(f)\cap U(\fg') \le I(f|_{\fg'}).
  \end{equation}
\end{lemma}
\begin{proof}
  Because we are working with {\it nilpotent} Lie algebras, twisted induction is just plain induction (\Cref{re:0twist}). Choose a polarization $\fh$ for $f$, so that
  \begin{equation*}
    I(f) = \ker\Ind_{\fh}^{\fg}(f). 
  \end{equation*}
  By \Cref{le:indres}, its intersection with $U(\fg')$ (i.e. the left-hand side of \Cref{eq:resnil}) is contained in the kernel of the induced representation $\Ind_{\fh\cap \fg'}^{\fg'}(f|_{\fh\cap \fg'})$. But because
  \begin{equation*}
    \fh\cap \fg'\le \fg'
  \end{equation*}
  is subordinate to $f|_{\fg'}$, that kernel is in turn contained in the right-hand side of \Cref{eq:resnil} by \cite[Lemma 6.4.3]{dixenv}.
\end{proof}

\begin{lemma}\label{le:tensnil}
  Let $\fg$ be a finite-dimensional nilpotent Lie algebra and $f,f'\in \fg^*$ two functionals with respective polarizations $\fh,\fh'\le \fg$. 

  We then have the inclusion
  \begin{equation}\label{eq:tensnil}
    \ker\left(\Ind_{\fh}^{\fg}(f)\otimes \Ind_{\fh'}^{\fg}(f')\right) \le I(f+f').
  \end{equation}
\end{lemma}
\begin{proof}
  This is a fairly straightforward application of \Cref{le:resnil} to the diagonal inclusion $\fg\le \fg\oplus \fg$:
  \begin{itemize}
  \item The enveloping algebra $U(\fg\oplus \fg)$ is the tensor square $U(\fg)^{\otimes 2}$ \cite[Proposition 2.2.10]{dixenv}.
  \item The subalgebra
    \begin{equation*}
      \fh\oplus \fh'\le \fg\oplus \fg
    \end{equation*}
    is a polarization for $f+f'\in \fg^*\oplus \fg^*$.
  \item Induction plays well with external tensor products:
    \begin{equation*}
      \Ind_{\fh\oplus \fh'}^{\fg\oplus \fg}(f+f')\cong \Ind_{\fh}^{\fg}(f)\otimes \Ind_{\fh'}^{\fg}(f')
    \end{equation*}
    as modules over
    \begin{equation*}
      U(\fg\oplus \fg)\cong U(\fg)\otimes U(\fg). 
    \end{equation*}
  \item And finally, the {\it internal} tensor product appearing on the left-hand side of \Cref{eq:tensnil} is the restriction of that same tensor product regarded as a $(U(\fg)\otimes U(\fg))$-module along the comultiplication \cite[\S 2.7.1]{dixenv}
    \begin{equation*}
      U(\fg)\to U(\fg)\otimes U(\fg)
    \end{equation*}
    that lifts the diagonal embedding $\fg\to \fg\oplus \fg$. 
  \end{itemize}
  This proves the claim. 
\end{proof}

\begin{remark}\label{re:add-vs-inv}
  \Cref{le:tensnil} is an additivity result of sorts for the map
  \begin{equation*}
    \fg^*\ni f\mapsto I(f)\in \Prim(U(\fg))
  \end{equation*}
  for nilpotent $\fg$ (where ``addition'' on the right-hand side corresponds to tensoring representations). The same map (under the same hypotheses) is also compatible with respect to ``taking inverses'': according to \cite[Lemme 8.1]{duflo-nil}, for nilpotent $\fg$ we have
  \begin{equation*}
    I(-f) = S(I(f)),
  \end{equation*}
  where $S$ is the antipode of the enveloping algebra $U:=U(\fg)$. If $I$ is the kernel of a $U$-representation $\rho$ then $S(I)$ is that of the dual $\rho^*$: the representation-theoretic analogue of an ``inverse''. 
\end{remark}

\pf{th:nil}
\begin{th:nil}
  
  The classification of the primitive ideals of $U:=U(\fg)$ via coadjoint orbits outlined in \cite[\S 6.1.5]{dixenv} and recalled in \Cref{con:primsolv} goes through. By \Cref{le:tensnil} the composite map
  \begin{equation*}
    \fg^*\ni f\mapsto I(f)\mapsto g_{I(f)}\in \cC(\fg)
  \end{equation*}
  is additive, and since it sends $0\in\fg^*$ to the trivial element it must in fact be a group morphism. Since on the other hand it also factors through the orbit space $\fg^*/\cA$ for the action of the algebraic adjoint group $\cA$ of $\fg$ (\Cref{con:primsolv}), it must descend to a group morphism
  \begin{equation}\label{eq:gacg}
    \fg^*_{\cA}\to \cC(\fg)
  \end{equation}
  from the group (also vector space) of coinvariants in $\fg^*$ under the coadjoint action. That morphism, composed with \Cref{eq:canlie}, is nothing but the usual identification
  \begin{equation}\label{eq:dualinv}
    \fg^*_{\cA}\cong \left(\fg^{\cA}\right)^*\cong \fz^*,\quad \fz:=\text{the center of }\fg,
  \end{equation}
  with the dual to the vector space $\fg^{\cA}\cong \fz$ of $\cA$-{\it invariants} in $\fg$ (i.e. the latter's center). Since \Cref{eq:gacg} is surjective and its composition with \Cref{eq:canlie} is the isomorphism \Cref{eq:dualinv}, \Cref{eq:canlie} itself must be an isomorphism.
\end{th:nil}

\subsection{Solvable}\label{subse:solv}

The following result supersedes \Cref{th:nil}, but that earlier argument is useful to have as a reference. 

\begin{theorem}\label{th:solv}
  Let $\fg$ be a finite-dimensional solvable Lie algebra over the algebraically closed field $\Bbbk$ of characteristic zero.

  The canonical map \Cref{eq:canlie} is an isomorphism.
\end{theorem}

Once more, some preparatory remarks are necessary.

\begin{lemma}\label{le:shift}
  Let
  \begin{itemize}
  \item $\fg$ be a finite-dimensional solvable Lie algebra over the algebraically closed field $\Bbbk$;
  \item $f,\lambda\in \fg^*$ with $\lambda$ annihilating $[\fg,\fg]$;
  \item and $\rho_f$ and $\rho_{f+\lambda}$ irreducible representations as in \Cref{not:irrif}. 
  \end{itemize}
  We then have
  \begin{equation}\label{eq:rrl}
    \rho_{f+\lambda}\preceq \rho_f\otimes\lambda,
  \end{equation}
  where $\lambda:\fg\to \Bbbk$ is regarded as a 1-dimensional representation.
\end{lemma}
\begin{proof}
  The claim is that the kernel $I(f+\lambda)$ of $\rho_{f+\lambda}$ contains that of the right-hand side of \Cref{eq:rrl}.

  Let $\fh\le \fg$ be a polarization for $f$, so that
  \begin{equation*}
    I(f)=\ker\widetilde{\Ind}_{\fh}^{\fg}(f) = \ker\Ind_{\fh}^{\fg}(f\otimes \theta_{\fg,\fh})
  \end{equation*}
  (see \Cref{con:primsolv}). The same Lie algebra $\fh$ is then also subordinate to $f+\lambda$ (because $\lambda$ by assumption vanishes on $[\fg,\fg]\ge [\fh,\fh]$), so by \cite[Lemma 6.4.2]{dixenv} we have
  \begin{equation}\label{eq:indindi}
    \ker\Ind_{\fh}^{\fg}(f\otimes\theta_{\fg,\fh}\otimes\lambda) = \ker\widetilde{\Ind}_{\fh}^{\fg}(f+\lambda)\le I(f+\lambda)
  \end{equation}
  Now, because the $\fh$-representation $\lambda$ is restricted from $\fg$, the push-pull formula for induction/restriction of Hopf-algebra modules shows that the leftmost induced representation in \Cref{eq:indindi} is
  \begin{equation*}
    \Ind_{\fh}^{\fg}(f\otimes\theta_{\fg,\fh})\otimes\lambda:
  \end{equation*}
  Apply \Cref{le:hopfpp} with
  \begin{itemize}
  \item $H=U(\fg)$ and $H'=U(\fh)$;
  \item with modules/representations $W=f\otimes\theta_{\fg,\fh}$ and $V=\lambda$.
  \end{itemize}
  We thus have
  \begin{equation*}
    \rho_{f+\lambda}\preceq \Ind_{\fh}^{\fg}(f\otimes\theta_{\fg,\fh})\otimes\lambda,
  \end{equation*}
  hence \Cref{eq:rrl}, via \Cref{pr:opscont} \Cref{item:8}, since the left-hand tensorands have the same kernel.
\end{proof}

\begin{lemma}\label{le:hopfpp}
  Let $H'\le H$ be an inclusion of Hopf algebras over an arbitrary field, $W$ a left $H'$-module and $V$ a left $H$-module. We then have an isomorphism
  \begin{equation*}
    H\otimes_{H'}(W\otimes V)\cong \left(H\otimes_{H'}W\right)\otimes V
  \end{equation*}
  of left $H$-modules.
\end{lemma}
\begin{proof}
  Using {\it Sweedler notation} $x\mapsto x_1\otimes x_2$ for Hopf algebra comultiplications (\cite[Notation 1.4.2]{mont}) and `$S$' for antipodes, we leave it to the reader to check that
  \begin{equation*}
    \begin{tikzpicture}[auto,baseline=(current  bounding  box.center)]
      \path[anchor=base] 
      (0,0) node (l) {$H\otimes_{H'}(W\otimes V)$}
      +(8,0) node (r) {$\left(H\otimes_{H'}W\right)\otimes V$}
      ;
      \draw[->] (l) to[bend left=6] node[pos=.5,auto] {$\scriptstyle h\otimes w\otimes v\longmapsto h_1\otimes w\otimes h_2v$} (r);
      \draw[->] (r) to[bend left=6] node[pos=.5,auto] {$\scriptstyle h_1\otimes w\otimes S(h_2)v\longmapsfrom h\otimes w\otimes v$} (l);
    \end{tikzpicture}
  \end{equation*}
  are mutually inverse module morphisms.
\end{proof}

Note, incidentally, the following consequence of \Cref{le:hopfpp} on the relation between $\Ind$ and $\widetilde{\Ind}$. The statement refers to the {\it nilradical} $\fn\le \fg$, i.e. the largest nilpotent ideal of $\fg$ (discussed e.g. in \cite[Proposition 1.4.9]{dixenv}). When $\fg$ is solvable the nilradical contains the derived ideal $[\fg,\fg]$ \cite[Corollary V.5.3]{ser-lalg}.

\begin{lemma}\label{le:twtens}
  
  Let $\fh\le \fg$ be an inclusion of finite-dimensional solvable Lie algebras over a field $\Bbbk$ of characteristic zero and $\rho:\fh\to \End(W)$ an $\fh$-representation.
  \begin{enumerate}[(1)]
  \item\label{item:1} If $\sigma\in\fh^*$ is a functional vanishing on $\fh\cap [\fg,\fg]$ then
    \begin{equation*}
      \widetilde{\Ind}_{\fh}^{\fg}(\rho\otimes\sigma) \cong \Ind_{\fh}^{\fg}(\rho)\otimes \lambda
    \end{equation*}
    for some functional $\lambda\in \fg^*$ annihilating $[\fg,\fg]+(\fn\cap \ker \sigma)$, where $\fn:=\fn(\fg)$ is the nilradical.
  \item\label{item:22} In particular,
    \begin{equation*}
      \widetilde{\Ind}_{\fh}^{\fg}(\rho) \cong \Ind_{\fh}^{\fg}(\rho)\otimes \lambda
    \end{equation*}
    for some $\lambda\in \fg^*$ annihilating the nilradical $\fn\trianglelefteq \fg$.
  \end{enumerate}
\end{lemma}
\begin{proof}
  Part \Cref{item:22} is indeed an instance of \Cref{item:1}: simply take $\sigma=0$. We thus focus on \Cref{item:1}.

  By definition (\Cref{not:induced}), we have
  \begin{equation*}
    \widetilde{\Ind}_{\fh}^{\fg}(\rho\otimes\sigma) \cong \Ind_{\fh}^{\fg}(\rho\otimes\theta_{\fg,\fh}\otimes\sigma). 
  \end{equation*}
  All eigenvalues of $ad(x)$ for $x\in \fn$ vanish, so the trace defining $\theta_{\fg,\fh}\in \fh^*$ vanishes on $\fh\cap \fn$. It follows that $\theta_{\fg,\fh}$ and $\sigma$ both vanish on
  \begin{equation*}
    \fh \cap \left([\fg,\fg]+(\fn\cap\ker\sigma)\right)
    =
    (\fh\cap[\fg,\fg]) + (\fn\cap\ker\sigma),
  \end{equation*}
  so $\theta_{\fg,\fh}+\lambda$ can be extended to a functional $\lambda\in\fg^*$ vanishing on the target space $[\fg,\fg]+(\fn\cap\ker\sigma)$. In particular $\lambda$ can be regarded as a 1-dimensional $\fg$-representation, and the conclusion then follows from the push-pull formula again: \Cref{le:hopfpp} with
  \begin{equation*}
    H=U(\fg),\ H'=U(\fh),\ V=\text{the 1-dimensional $\fg$-module attached to }\lambda.
  \end{equation*}
\end{proof}

Solvable analogues of \Cref{le:resnil,le:tensnil}, needed below:

\begin{lemma}\label{le:indrestw}
  Let $\fg$ be a finite-dimensional solvable Lie algebra over an algebraically closed field of characteristic zero.
  \begin{enumerate}[(1)]
  \item\label{item:20} Let $\fg'\le \fg$ be a Lie subalgebra, $f\in \fg^*$ and $f':=f|_{\fg}$. Recalling \Cref{not:irrif}, we have
    \begin{equation*}
      \ker \rho_f \cap U(\fg')\le \ker\left(\rho_{f'}\otimes\lambda\right)
    \end{equation*}
    for some 1-dimensional representation $\lambda\in (\fg')^*$ annihilating the intersection $\fg'\cap \fn$ with the nilradical $\fn:=\fn(\fg)\trianglelefteq \fg$. 
  \item\label{item:21} For two functionals $f,f'\in \fg^*$ we have
    \begin{equation*}
      \ker\left(\rho_f\otimes\rho_{f'}\right) \le \ker\left(\rho_{f+f'}\otimes\lambda\right)
    \end{equation*}
    for some 1-dimensional representation $\lambda\in (\fg)^*$ annihilating the nilradical of $\fg$.     
  \end{enumerate}  
\end{lemma}
\begin{proof}
  The arguments are minor adaptations of those respectively employed in the proofs of \Cref{le:resnil,le:tensnil}. 
  \begin{enumerate}[(1)]
  \item Following the same line of reasoning as in \Cref{le:resnil}: choose a polarization $\fh\le \fg$ for $f$, giving
    \begin{equation*}
      I(f) = \ker\Ind_{\fh}^{\fg}(f\otimes\theta_{\fg,\fh}).
    \end{equation*}
    Then, by \Cref{le:indres} again,
    \begin{equation}\label{eq:iukerker}
      I(f)\cap U(\fg')
      \le
      \ker \Ind_{\fh\cap \fg'}^{\fg'}(f\otimes\theta_{\fg,\fh})
      =
      \ker \widetilde{\Ind}_{\fh\cap \fg'}^{\fg'}(f\otimes\theta_{\fg,\fh}\otimes\theta_{\fg',\fh\cap \fg'}^*),
    \end{equation}
    where as usual, the `$*$' superscript denotes the dual representation.

    Now note that both $\theta_{\fg,\fh}|_{\fh\cap \fg'}$ and $\theta_{\fg',\fh\cap \fg'}$ vanish on $\fh\cap \fg'\cap \fn$, so they extend to functionals on $\fg'$ vanishing on $[\fg',\fg']+(\fg'\cap \fn)$. But now, by \Cref{le:twtens} \Cref{item:1}, the rightmost representation in \Cref{eq:iukerker} is
    \begin{equation*}
      \widetilde{\Ind}_{\fh\cap \fg'}^{\fg'}(f\otimes\theta_{\fg,\fh}\otimes\theta_{\fg',\fh\cap \fg'}^*)
      \cong
      \widetilde{\Ind}_{\fh\cap \fg'}^{\fg'}(f)\otimes\lambda
    \end{equation*}
    for some $\lambda\in (\fg')^*$ as in the statement.
  \item This follows from part \Cref{item:20} the same way \Cref{le:tensnil} follows from \Cref{le:resnil}, taking into account the fact that for the diagonal embedding $\fg\le \fg\oplus \fg$ the intersection
    \begin{equation*}
      \fg\cap \fn(\fg\oplus \fg)
    \end{equation*}
    is nothing but the nilradical $\fn(\fg)$. 
  \end{enumerate}
  This concludes the proof of the two claims.
\end{proof}

We will prove \Cref{th:solv} by induction on the dimension of $\fg$, with the following result carrying the brunt of the iterative load. 

\begin{proposition}\label{pr:solvindstep}
  Let $\fg$ be a finite-dimensional solvable Lie algebra over the algebraically closed field $\Bbbk$, and assume the conclusion of \Cref{th:solv} holds for all solvable Lie algebras of smaller dimension.

  The following conditions are equivalent.
  \begin{enumerate}[(a)]
  \item\label{item:15} \Cref{th:solv} holds for $\fg$.
  \item\label{item:16} The canonical map \Cref{eq:canlie} is injective.   
  \item\label{item:17} For every 1-dimensional $\fg$-representation $\lambda\in \fg^*$ annihilating the center $\fz:=\fz(\fg)$ the generator $g_{\lambda}\in \cC(\fg)$ is trivial.
  \item\label{item:18} Same as \Cref{item:17}, but only for those 1-dimensional representations $\lambda\in \fg^*$ annihilating the nilradical $\fn\trianglelefteq \fg$.
  \end{enumerate}
\end{proposition}
\begin{proof}
  That \Cref{item:15} and \Cref{item:16} are equivalent is a general remark (\Cref{le:cansurj}). Naturally, \Cref{item:16} implies \Cref{item:17}, and the latter is formally stronger than \Cref{item:18} (since the center is contained in the nilradical, so {\it fewer} functionals will annihilate it). It thus remains to prove \Cref{item:18} $\Longrightarrow$ \Cref{item:15}, which implication we henceforth focus on.
  
  By \Cref{le:indrestw} \Cref{item:21}, \Cref{le:tensnil} holds with the caveat that one might have to tensor by 1-dimensional representations $\lambda\in \fg^*$ annihilating the nilradical $\fn\trianglelefteq \fg$. Since we are assuming
  \begin{equation*}
    1=g_{\lambda}\in \cC(\fg)\text{ for all such }\lambda,
  \end{equation*}
  we can conclude as in the proof of \Cref{th:nil}: the map
  \begin{equation*}
    \fg^*\ni f\mapsto I(f)\mapsto g_{I(f)}\in \cC(\fg)
  \end{equation*}
  is additive, gives an isomorphism \Cref{eq:dualinv} upon further composition with \Cref{eq:canlie}, and we are done. 
\end{proof}

\pf{th:solv}
\begin{th:solv}
  As announced, we proceed by induction on $\dim\fg$, with the base case(s) being simple exercises. This leaves the induction step which, per \Cref{pr:solvindstep}, amounts to showing that
  \begin{equation*}
    1=g_{\lambda}\in \cC(\fg),\ \forall \lambda\in \fn^{\perp}\le \fg^*
  \end{equation*}
  (where as before, $\fn\trianglelefteq \fg$ is the nilradical).

  Fix such a functional $\lambda\in\fn^{\perp}$ (non-zero, or there is nothing to prove). Because the chain group is functorial for surjections, the induction hypothesis (together with \Cref{pr:solvindstep}) tells us that $\lambda$ cannot annihilate the center of any proper quotient of $\fg$. Since $\ker\lambda$ contains some 1-dimensional ideal
  \begin{equation*}
    \fk:=\Bbbk z\trianglelefteq \fg,
  \end{equation*}
  because we are working over an algebraically closed field (e.g. \cite[1.3.12]{dixenv}). It follows that the quotient $\fg/\fk$ splits as
  \begin{equation*}
    \fg/\fk
    =
    \left(\ker \lambda|_{\fg/\fk}\right)
    \oplus
    \left(\text{image of }\Bbbk x\right)
  \end{equation*}
  for some $x\in \fg$ not annihilated by $\lambda$.

  If $x$ were to commute with $z$ then $\lambda$ would fail to annihilate the nilpotent ideal $\spn\{x,z\}$, contradicting the choice $\lambda\in \fn^{\perp}$. We can thus assume that $[x,z]=z$; because $z$ commutes with $\ker\lambda$ modulo $z$, we have a decomposition
  \begin{equation*}
    \ker\lambda = \Bbbk z\oplus \fy,\ \fy:=C_{\ker\lambda}(x) = \{x'\in \ker\lambda\ |\ [x,x']=0\}. 
  \end{equation*}
  There are now some possibilities to consider:
  \begin{enumerate}[(a)]
  \item {\bf The Lie algebra $\fy$ centralizes $z$.} In this case $\fy$ is an ideal in $\fg$, and the quotient $\fg/\fy$ is (isomorphic to) the non-nilpotent 2-dimensional Lie algebra generated by $x$ and $z$. This is the $\fg_2$ of \cite[Lemma 6.1.2 (i)]{dixenv}, and has faithful irreducible representations.

    It follows from \Cref{pr:faiths} that {\it all} irreducible $(\fg/\fy)$-representations (in particular $\lambda$) are trivial in the chain group, hence also in $\cC(\fg)$ by the above-mentioned chain-group functoriality under quotients.
  \item {\bf $[\fy,z]\ne \{0\}$.} In this case there is some $y\in \fy\le \ker\lambda$ with $[y,z]=z$, and $\lambda$ fails to annihilate the $\fg$-central element $x-y$. This contradicts our assumption that $\lambda\in\fn^{\perp}$, and finishes the proof.  \qedhere
  \end{enumerate}
\end{th:solv}

\subsection{Semisimple}\label{subse:ss}

The main result we address here is

\begin{theorem}\label{th:ss}
  Let $\fg$ be a finite-dimensional, semisimple Lie algebra over the algebraically closed field $\Bbbk$ of characteristic zero.

  The chain group $\cC(\fg)$ is trivial, and hence \Cref{eq:canlie} is an isomorphism.
\end{theorem}

We use some of the language familiar in the theory of semisimple Lie algebras as covered, say, in \cite{hmph}, \cite[Chapter 1]{dixenv}, etc. In particular:
\begin{itemize}
\item $\fh\le \fg$ will be a {\it Cartan subalgebra} of $\fg$ (\cite[\S 15]{hmph}, \cite[\S 1.9]{dixenv}).
\item This then induces a {\it root-space decomposition} for $\fg$ \cite[\S 8]{hmph}, for which we assume we have chosen a {\it base} $\Delta\subset\fh^*$ \cite[\S 10.1]{hmph}.
\item We denote by $\delta\in \fh^*$ the half-sum of the {\it positive roots} \cite[\S 10.1]{hmph} attached to the choice of $\Delta$. This element is discussed in \cite[\S 13.3]{hmph} as well as \cite[\S 11.1.13]{dixenv}, where it is denoted by the same symbol.
\item Let $\lambda\in \fh^*$. Following \cite[\S 7.1.4]{dixenv} or \cite[\S I.1]{dufl-ss}, we denote by $M(\lambda)$ the {\it Verma module} of highest weight $f-\delta$. For comparison: in \cite[\S 20.3]{hmph} $M(\lambda)$ would rather be denoted by $Z(\lambda-\delta)$.
\item The simple quotient of $M(\lambda)$ is $L(\lambda)$, again following either \cite[\S 7.1.12]{dixenv} or \cite[\S I.1]{dufl-ss}. \cite[\S 20.3]{hmph} would set $L(\lambda)=V(\lambda-\delta)$.
\item We write $W$ for the {\it Weyl group} of $\fg$ (\cite[\S 10.3]{hmph}, \cite[\S 1.10.10]{dixenv}); it is a finite group of linear automorphisms of $\fh^*$, generated by reflections.
\end{itemize}

\pf{th:ss}
\begin{th:ss}
  The notation outlined above is in force throughout. We will also (somewhat abusively) identify representations with their carrier spaces (e.g. by referring to a representation of $U:=U(\fg)$ on $V$ as just plain $V$, regarded as a $U$-module). Finally, recall from \Cref{res:primenough} \Cref{item:12} that we may as well attach generators $g_J$ of the chain group to primitive ideals $J\le U$ (rather than actual representations); we do this below.

  According to \cite[Theorem 8.4.4 (iv)]{dixenv}, the respective kernels $J_{\lambda}$ of $M(\lambda)$, $\lambda\in \fh^*$ are precisely the minimal primitive ideals of $U$. An arbitrary primitive ideal $J$ will thus contain some $J_{\lambda}$, and by \Cref{le:incsame} we have  
  \begin{equation}\label{eq:contsame}
    J\ge J_{\lambda}\Rightarrow g_J = g_{J_{\lambda}}\in \cC(\fg). 
  \end{equation}
  
  \cite[Theorem 8.4.4 (ii)]{dixenv} moreover shows that $J_{\lambda}=J_{\lambda'}$ whenever $\lambda'$ is in the Weyl-group orbit $W\lambda$; the same must be true of chain-group generators then:
  \begin{equation*}
    g_{J_{\lambda}}=g_{J_{\lambda'}},\ \forall \lambda'\in W\lambda. 
  \end{equation*}
  One last observation: the tensor product $M(\lambda)\otimes M(\lambda')$ (for arbitrary $\lambda,\lambda'\in \fh^*$) contains a highest-weight vector of weight $\lambda+\lambda'-2\delta$, so it has a submodule surjecting onto the simple module $L(\lambda+\lambda'-\delta)$. Because the latter's kernel contains $J_{\lambda+\lambda'-\delta}$ (and hence is identified to it in $\cC(\fg)$ by \Cref{eq:contsame}), and writing $g_{\lambda}:=g_{J_{\lambda}}$ for better readability, we have
  \begin{equation}\label{eq:delshitfadd}
    g_{{\lambda}} g_{{\lambda'}} = g_{{\lambda+\lambda'-\delta}},\ \forall \lambda,\lambda'\in \fh^*.
  \end{equation}
  In summary:
  \begin{itemize}
  \item The map
    \begin{equation*}
      \fh^*\ni \lambda\mapsto g_{\lambda}:=g_{J_{\lambda}}\in \cC(\fg)
    \end{equation*}
    is onto;
  \item and invariant under the Weyl-group action:
    \begin{equation}\label{eq:winv}
      g_{w\lambda}=g_{\lambda},\ \forall w\in W,\ \forall \lambda\in \fh^*;
    \end{equation}
  \item and ``$\delta$-shifted-additive'' in the sense of \Cref{eq:delshitfadd}.
  \end{itemize}
  Applying \Cref{eq:delshitfadd} to $w\lambda$ and $w\lambda'$ instead, using \Cref{eq:winv} once and absorbing $\lambda+\lambda'$ into a single $\lambda$, we obtain
  \begin{equation}\label{eq:eqeq}
    g_{\lambda-w^{-1}\delta} = g_{w\lambda-\delta} = g_{\lambda-\delta},\ \forall w\in W,\ \forall \lambda\in \fh^*.
  \end{equation}  
  Because the longest element $w_0\in W$ (\cite[\S 7.2.3]{dixenv}; this is the $\sigma$ of \cite[Exercise 10.9]{hmph}) sends $\delta$ to $-\delta$, the leftmost term of \Cref{eq:eqeq} can be set to $g_{\lambda+\delta}$. But then the right-hand side of \Cref{eq:delshitfadd} also holds with `$+\delta$' in place of `$-\delta$', whereupon the substitutions $\lambda\mapsto \lambda-\delta$ and $\lambda'\mapsto \lambda'-\delta$ yield
  \begin{equation*}
    g_{\lambda-\delta}g_{\lambda'-\delta} = g_{\lambda+\lambda'-\delta},\ \forall \lambda,\lambda'\in \fh^*.  
  \end{equation*}
  Together with \Cref{eq:eqeq}, this finally tells us that
  \begin{equation*}
    \fh^*\ni \lambda\mapsto g_{\lambda-\delta}\in \cC(\fg)
  \end{equation*}
  descends to a surjection from the group
  \begin{equation*}
    \fh^*/\langle w\lambda-\lambda,\ w\in W,\ \lambda\in \fh^*\rangle.
  \end{equation*}
  of coinvariants under the action of the Weyl group. That group of coinvariants is of course trivial, because the representation of $W$ on $\fh^*$ contains no copies of the trivial representation (obviously: by definition, $W$ is generated by non-trivial reflections).
\end{th:ss}

\begin{remark}
  Note, incidentally, that in none of the proofs thus far have we had to make direct use of dual representations (via the binary relation `$\sim$' of \Cref{def:chainlie}): additivity in the form of \Cref{eq:rrr} was enough, essentially for the familiar reason that if a semigroup happens to be a group, that group structure is unique.
\end{remark}



\addcontentsline{toc}{section}{References}

\Addresses

\end{document}